\documentclass[11pt]{amsart}
\usepackage{amsfonts}
\usepackage{fullpage, amsmath, amsthm,amsfonts,amssymb,stmaryrd, mathrsfs,amscd}\usepackage{graphicx}
\usepackage{hyperref}
\usepackage{cleveref}
\makeatletter
\usepackage[pdftex]{color}
\usepackage[all, 2cell]{xy}
\UseAllTwocells

\usepackage{amsmath}
\usepackage{amssymb}
\usepackage{fullpage, amsmath, amsthm,amsfonts,amssymb,stmaryrd, mathrsfs,amscd}\usepackage{graphicx}
\usepackage{xcolor}

\numberwithin{equation}{section}

\newtheorem{thm}{Theorem}[section]

\newtheorem{prop}[thm]{Proposition}
\newtheorem{lem}[thm]{Lemma}
\newtheorem{cor}[thm]{Corollary}

\newtheorem{rem}[thm]{Remark}

\newcommand{\frakP}{{\mathfrak P}}

\newcommand{\bB}{{\mathbb B}}

\newcommand{\bG}{{\mathbb G}}

\newcommand{\bQ}{{\mathbb Q}}
\newcommand{\bR}{{\mathbb R}}

\newcommand{\bX}{{\mathbb X}}

\newcommand{\bZ}{{\mathbb Z}}

\newcommand{\bfA}{{\mathbf A}}
\newcommand{\bfB}{{\mathbf B}}
\newcommand{\bfC}{{\mathbf C}}

\newcommand{\bfa}{{\mathbf a}}

\newcommand{\bff}{{\mathbf f}}

\newcommand{\mA}{{\mathcal A}}

\newcommand{\mF}{{\mathcal F}}
\newcommand{\mG}{{\mathcal G}}
\newcommand{\mH}{{\mathcal H}}
\newcommand{\mI}{{\mathcal I}}

\newcommand{\mM}{{\mathcal M}}
\newcommand{\mN}{{\mathcal N}}
\newcommand{\mO}{{\mathcal O}}
\newcommand{\mP}{{\mathcal P}}

\newcommand{\mS}{{\mathcal S}}
\newcommand{\mT}{{\mathcal T}}

\newcommand{\mY}{{\mathcal Y}}

\newcommand{\scrA}{{\mathscr A}}
\newcommand{\scrB}{{\mathscr B}}

\newcommand{\al}{{\alpha}}
\newcommand{\be}{{\beta}}
\newcommand{\ga}{{\gamma}}
\newcommand{\Ga}{{\Gamma}}

\newcommand{\La}{{\Lambda}}

\newcommand{\Alg}{\mathrm{Alg}}

\newcommand{\Mod}{\mathrm{Mod}}

\newcommand{\lincat}{\mathrm{Lincat}}

\newcommand{\colim}{\mathrm{colim}}

\newcommand{\Hom}{\mathrm{Hom}}

\newcommand{\Aut}{\mathrm{Aut}}

\newcommand{\Ind}{\mathrm{Ind}}

\newcommand{\Res}{\mathrm{Res}}

\newcommand{\Spec}{\mathrm{Spec}}
\newcommand{\prestk}{\mathrm{PreStk}}
\newcommand{\corr}{\mathrm{Corr}}

\newcommand{\unip}{\mathrm{unip}}

\newcommand{\ad}{\mathrm{ad}}

\newcommand{\af}{\mathrm{aff}}

\newcommand{\Gr}{\mathrm{Gr}}

\newcommand{\xcoch}{\bX_\bullet}

\newcommand{\padic}{\bQ_p}

\newcommand{\ab}{\mathrm{ab}}
\newcommand{\mon}{\mathrm{mon}}

\newcommand{\bs}{\backslash}
\newcommand{\Shv}{\mathrm{Shv}}
\newcommand{\Gal}{\mathrm{Gal}}
\newcommand{\ext}{\mathrm{ext}}
\newcommand{\s}{\mathrm{sc}}
\newcommand\ov{\overline}
\newcommand\id{\mathrm{id}}
\newcommand\wt{\widetilde}
\newcommand\un{\underline}
\newcommand\br{\breve}

\newcommand\Int{\textup{Int}}
\newcommand{\Ximon}{\Xi\textup{-mon}}

\newcommand{\RG}{\mathrm{TPRG}}

\begin{document}

\markboth{\hfill{\rm Zhiwei Yun and Xinwen Zhu} \hfill}{\hfill {\rm Affine Hecke categories in equal and mixed characteristic \hfill}}

\title{Affine Hecke categories in equal and mixed characteristic}

\author{Zhiwei Yun and Xinwen Zhu}

\begin{abstract}
For a quasi-split tamely connected reductive group $G$ over a $p$-adic field, we prove that its (monodromic) affine Hecke category is canonically equivalent to its equal characteristic counterpart as monoidal categories.
\end{abstract}

\maketitle

\setcounter{tocdepth}{1}
\tableofcontents

\section{Introduction}\label{sec: introduction}

Let $G$ be a split connected reductive group over a Laurent series field $k(\!(u)\!)$.
Affine Hecke categories of $G$ are ubiquitous in modern geometric representation theory and related fields. The celebrated Bezrukavnikov equivalence provides a coherent realization of the unipotent affine Hecke category of $G$ in terms of the Langlands dual group $G^{\vee}$. This can be seen as the unipotent part of the local geometric Langlands correspondence, which has recently found numerous applications in diverse areas such as representation theory, algebraic geometry and knot theory.

For applications to certain questions in number theory, it is highly desirable to develop parallel theories for affine Hecke categories of a quasi-split connected reductive group $G$ over a $p$-adic field $F$. Here, a $p$-adic field is understood to be a characteristic zero, discretely valued complete local field whose residue field is a perfect field of characteristic $p>0$. In particular, establishing a version of Bezrukavnikov equivalence in this context is of significant interest. Recently, considerable progress has been made towards this problem, as seen in \cite{ALWY} and \cite{Ban23}. Notably, \cite{Ban23} constructed the so-called Fargues-Fontaine surface and used it to construct a deformation of the mixed characteristic affine flag variety to the equal characteristic affine flag variety. By utilizing this device, he proved that the mixed and the equal characteristic (unipotent) affine Hecke categories are equivalent, and hence established the Berukavnikov equivalence and the derived geometric Satake in mixed characteristic.

In this note, we reprove and generalize Bando's results by an entirely different method. Let $F$ be a finite extension of $\padic$, with $\mO$ its ring of integers, $\kappa$ its residue field and $\varpi\in \mO$ a uniformizer.  Let $k$ be a fixed algebraic closure of $\kappa$. Let $\breve F$ be the completion of the maximal unramified extension of $F$ with residue field $k$; denote by $\breve \mO$ the ring of integers in $\breve F$.

Let $G$ be a tamely ramified quasi-split connected reductive group over $F$, with a fixed pinning $(B,T,e: U\to \bG_a)$, where as usual $B$ denotes a Borel subgroup of $G$, $T$ a maximal torus of $B$, $U$ the unipotent radical of $B$ and $e: U\to \bG_a$ a homomorphism from $U$ to the additive group that is non-degenerate (in the sense that it is nontrivial when restricted to each simple root subgroup, when base changed to $\overline F$). 

The pinning $(B,T,e)$ determines an Iwahori group scheme $\mI$ of $G$ over $\mO$ (see \S \ref{Sec: unip}). Let $I=L^+\mI$ be the associated positive loop group over $\kappa$ as defined in \cite{Zhu17}. We recall that it is the perfection of the Greenberg realization $L^+_p\mI$, which assigns every $\kappa$-algebra $R$ the group $\mI(W_\mO(R))$, where $W_\mO(R)=W(R)\otimes_{W(\kappa)}\mO$ and $W(-)$ is the ring of $p$-typical Witt vectors of $R$. Let $I^+\subset I$ denote the pro-unipotent radical of $I$. The quotient  $\mS^{\natural}_\kappa:=I/I^+$ is a finite dimensional torus over $\kappa$. On the other hand, let $LG$ be the loop group of $G$. Recall it is a functor defined on the category of perfect $\kappa$-algebras, sending $R$ to $G(W_\mO(R)[1/p])$. For a geometric object $\mY$ defined over $\kappa$, we use $\mY_{k}$ to denote its base change to $k$, e.g., $LG_{k}, I_{k}$ and $\mS^{\natural}_{k}$. 

Fix a prime $\ell$ different from $p$, and we work with \'etale sheaves of $\La:=\bZ_{\ell}$-modules.
Consider the unipotent affine Hecke category
\[
\mH_{\unip,\af}:=\Shv\bigl(I_k\bs LG_k/I_k,\La\bigr),
\]
equipped with an automorphism $\sigma_*$ given by $*$-pushforward along the Frobenius endomorphism of $\sigma:I_k\bs LG_k/I_k$.
As in \cite{DLYZ25, Zhu25}, one can also define the monodromic affine Hecke category
\[
\mH_{\mon,\af}:=\Shv\bigl((I_k,\mon)\bs LG_k/(I_k,\mon),\La\bigr),
\] 
again equipped with a Frobenius automorphism $\sigma_*$. We remark that $\mH_{\mon,\af}$ is a full subcategory of $\Shv\bigl(I^{+}_k\bs LG_k/I^{+}_k,\La\bigr)$, whose objects are locally constant along left and right $\mS^{\natural}_{k}$-orbits. If we fix a character sheaf $\chi$ on $\mS^{\natural}_k$ (i.e., a rank one tame local system on $\mS^{\natural}_{k}$ with a trivialization of its stalk at $1\in \mS^{\natural}_{k}$), we can also consider the corresponding equivariant category with respect to $\chi$ both on the left and on the right
\[
{}_{\chi}\mH_{\chi,\af}=(\Mod_\La)_\chi\otimes_{\Shv_{\mon}(\mS^{\natural}_k)}\mH_{\mon,\af}\otimes_{\Shv_{\mon}(\mS^{\natural}_k)}(\Mod_\La)_\chi.
\]
In particular, when $\chi=u$ is the trivial character sheaf, then 
\begin{equation}\label{eq: equiv vs mon}
\mH_{\unip,\af}\cong (\Mod_\La)_u\otimes_{\Shv_{\mon}(\mS^{\natural}_k)}\mH_{\mon,\af}\otimes_{\Shv_{\mon}(\mS^{\natural}_k)}(\Mod_\La)_u .
\end{equation}

Now, let $F^\flat=\kappa(\!(u)\!)$ be the field of formal Laurent series over $\kappa$ in the variable $u$. It is well-known that the tame Galois groups of $F$ and $F^\flat$ are canonically identified, denoted by $\Ga$.
We let $\RG_F$ (resp. $\RG_{F^\flat}$) denote the groupoid of tamely ramified pinned connected reductive groups over $F$ (resp. over $F^\flat$). Then both categories are canonically equivalent to the groupoid of based root data equipped with continuous actions of $\Ga$. Therefore, associated to $(G,B,T,e)$ there is a pinned connected reductive group $(G^\flat,B^\flat,T^\flat,e^\flat)$ over $F^\flat$, unique up to a unique isomorphism. We may consider the corresponding unipotent affine Hecke category 
\[
\mH^\flat_{\unip,\af}:=\Shv\bigl(I^\flat_k\bs LG^\flat_k/I^\flat_k,\La\bigr)
\]
and the monodromic affine Hecke category
\[
\mH^\flat_{\mon,\af}:=\Shv\bigl((I_k^\flat,\mon)\bs LG_k^\flat/(I_k^\flat,\mon), \La\bigr),
\]
both equipped with monoidal structures and an automorphism $\sigma_*
$ induced by pushforward along the Frobenius endomorphism. The main result of this note is as follows.

\begin{thm}\label{intro: main thm}  
There are canonical equivalences of monoidal categories
\[
\mH_{\mon,\af}\cong \mH_{\mon,\af}^\flat, \quad \mH_{\unip,\af}\cong \mH_{\unip,\af}^\flat,
\]
matching (co)standard objects, and compatible with the $\sigma_*$-actions.
\end{thm}

\begin{rem}
\begin{enumerate}
\item  It follows immediately that the equivalences in the theorem match the perverse $t$-structures (as sheaves on the affine flag varieties and its enhanced version $LG_k/I^+_k$) in mixed and equal characteristic.
\item  It is standard to deduce equivalences between parahoric affine Hecke categories in mixed and equal characteristic, matching the corresponding perverse $t$-structures. See \cite[Theorem 8.26]{DLYZ25}. In particular, when $G$ is unramified, the Satake category, a.k.a. the category of $L^+G$-equivariant perverse sheaves on the affine Grassmannain $\Gr_G$ of $G$, and its equal characteristic counterpart, are canonically matched as \emph{abelian monoidal} categories. In addition, the monoidal structures on the hypercohomology functor $H^*(\Gr_G,-)$ and $H^*(\Gr_{G^\flat},-)$ are also canonically matched. As explained in \cite[Proposition 2.21]{Zhu17}, the existence of commutativity constraints compatible with the monoidal structure on the hypercohomology functor is a property (rather than an additional structure) of the Satake category. We thus deduce that the Satake category in mixed and equal characteristic are equivalent as abelian symmetric monoidal categories. This reproves the main result of \cite{Zhu17} by a different method.
\end{enumerate}
\end{rem}

\medskip

\noindent \bf Acknowledgment. \rm The result of this note was announcement by one of us in conference talks a few years ago. We apologize for the delay of its appearance. 
The work of Z.Y. is partially supported by a Simons Investigator grant. The work of X.Z. is partially supported by NSF grant under DMS-2200940.

\section{Group schemes}\label{Sec: unip}
\subsection{Mixed characteristic}
Let $(G,B,T,e)$ be a pinned tamely ramified quasi-split connected reductive group over $F$. We set up a few notations that will be needed in the sequel. Suppose $G$ splits over a tamely ramified finite Galois extension $\tilde F/F$ (in a fixed algebraic closure $\overline{\breve F}$ of $\breve F$).   
Let $S\subset T$ denote the $F$-rational, maximally $\breve F$-split subtorus and let $A\subset S$ be the maximal $F$-split torus. Let $\mA\subset \mS\subset\mT$ be the (unique) Iwahori group schemes of $A\subset S\subset T$ over $\mO$. We let $\widetilde{W}=N_G(S)(\breve F)/\mS(\breve\mO)$ denote the Iwahori-Weyl group of $G_{\breve F}$. Let $W_\af\subset \widetilde W$ be the affine Weyl group of $G$. Recall that it can be defined as Iwahori-Weyl group of the simply-connected cover $G_{\s}$ of the derived subgroup of $G$ (associated to the preimage $T_\s$ of $T$ in $G_\s$). It is known that 
\begin{equation*}\label{eq-Omega-dual-group}
 \widetilde W/W_\af \cong \pi_1(G)_{\Gal(\tilde F\breve F/\breve F)}\cong \pi_0(LG_k).
\end{equation*}
Here $\pi_1(G)$ is the algebraic fundamental group of $G$, equipped with a continuous $\Gal(\ov F/F)$-action.

We let $\scrB(G, \breve F)$ denote the (reduced) Bruhat-Tits building of $G_{\breve F}$, which is a contractible (poly)simplicial complex equipped with an action of $G(\breve F)$ by simplicial automorphisms. We shall call the \emph{interior} of a cell in this simplicial complex a facet. Let $\scrA(G,S,\breve F)\subset \scrB(G,\breve F)$ denote the apartment corresponding to $S_{\breve F}$.
It is $N_G(S)(\breve F)$-stable, and the action of $N_G(S)(\breve F)$ on it factors through an action of $\widetilde{W}$ by affine transformations. Note that the $\sharp \kappa$-Frobenius $\sigma\in \Aut(\breve F/F)$ acts on everything. In particular, $\scrA(G, S, \breve F)^\sigma=\scrA(G,A,F)$ is the apartment associated to $A$ in the building $\scrB(G,F)=\scrB(G,\breve F)^\sigma$. 

The pinning $(B,T,e)$ determines a special vertex in $v_0\in \scrA(G,S,\breve F)$ and the corresponding parahoric group scheme $\mP_{v_0}$ as follows. Let $H$ be a split Chevalley group scheme over $\bZ$ such that $H_{\breve F}\cong G_{\breve F}$. Choose a pinning $(B_H,T_H,e_H)$ of $H$, and let $\Xi_H:=\Aut(H,B_H,T_H,e_H)$, viewed as a finite discrete group. 
Then 
\begin{equation}\label{eq: xiG}
\xi_G:= \mathrm{Isom}\bigl((H,B_H,T_H,e_H)_F, (G,B,T,e)\bigr)
\end{equation} 
is a $\Xi_H$-torsor over $\Spec F$ that has a $\tilde F$-point. Choosing such a point $\eta$ defines a homomorphism $\xi_{G,\eta}:  \Gal(\tilde F/F)\to \Xi_H$ such that the automorphism $\id_{G}\times \ga$ on $G_{\tilde F}=G\otimes_{F}\tilde F$ corresponds to the automorphism $\xi_{G,\eta}(\ga)\times \ga$ on $H_{\tilde F}=H\otimes_{\bZ}\tilde F$. 
Let $\tilde\mO$ be the ring of integers of $\tilde F$. Let
\[
\mP_{v_0}=(\Res_{\tilde\mO/\mO}H_{\tilde\mO})^{\Gal(\tilde F/F),\circ}.
\]
Here and below, $(-)^\circ$ denotes taking the neutral connected component. Via $\eta$, $\mP_{v_0}$ can be regarded as an integral model of $G$. As changing $\eta$ amounts to conjugating $\xi_{G,\eta}$ by elements in $\Xi_H$,  we see that $\mP_{v_0}$ as an integral model of $G$ is independent of the choice of $\eta$. This is the desired special parahoric of $G$.
Putting differently, the reductive group $H_{\mO}$ over $\mO$ defines a hyperspecial point $v_0$ in the apartment $\scrA(H,T_H,\breve F)$, fixed by the action of $\Xi_H$. The choice of $\eta$ identifies $\scrA(G,A,F)$ with $\scrA(H,T_H, \tilde F)^{\Gal(\tilde F/F)}$, which allows us to view $v_0$ as a point in $\scrA(G,A,F)\subset \scrA(G,S,\breve F)$.

Let $S_{\ad}$ be the image of $S$ in the adjoint group $G_\ad$ of $G$, so that $\scrA(G,S,\breve F)$ is a torsor under the real vector space $\xcoch(S_\ad)_\bR$ spanned by $\breve F$-rational cocharacters of $S_\ad$. We identify $\scrA(G,S,\breve F)$ with $\xcoch(S_\ad)_\bR$ so that $v_{0}$ corresponds to $0$. We then let $\breve\bfa$ be the (unique) alcove such that $v_0$ is contained in the closure of $\breve\bfa$ and $\breve\bfa$ is contained in the dominant Weyl chamber of $\xcoch(S_\ad)_\bR$ determined by the Borel subgroup $B$. Then the group $\widetilde W$ acquires a length function.
Let $\Omega\subset \widetilde W$ be the corresponding subgroup of length zero elements, or equivalently the stabilizer of the alcove $\breve\bfa$ under $\wt W$. Then 
\begin{equation}\label{eq-Iwahori-Weyl-semiproduct}
   \widetilde W=W_\af\rtimes\Omega. 
\end{equation}
In particular, $\Omega\cong \pi_1(G)_{\Gal(\tilde F\breve F/\breve F)}\cong \pi_0(LG)$.
Note that  $\mathbf{a}=\breve\bfa\cap \scrA(G,A,F)$ is an alcove of $\scrA(G,A,F)$, and the decomposition \eqref{eq-Iwahori-Weyl-semiproduct} is preserved under the action of $\sigma$.

For a facet $\breve\bff$ in $\scrA(G,S,\breve F)$, let $\breve\mP_{\breve\bff}$ be the corresponding parahoric group scheme over $\breve\mO$. If $\breve\bff$ is $\sigma$-stable, then $\breve\mP_{\breve\bff}$ descends to a group scheme $\mP_{\breve\bff}$ over $\mO$. Let $M_{\br\bff}$ denote the reductive quotient of the special fiber $\breve\mP_{\breve\bff}\otimes_{\br\mO}k$, which is a connected reductive group over $k$.

In particular, the alcove  $\br\bfa$ corresponds to an Iwahori group scheme $\br\mI$ over $\br\mO$, which descends to an Iwahori group scheme $\mI$ of $G$ over $\mO$. Let $I=L^{+}_{p}\mI$ be the $p$-adic positive loop group of $\mI$, which is a pro-algebraic group over $\kappa$. Then the base change $I_{k}$ is $L^{+}\br\mI$. Let $I^{+}$ (resp. $I^{+}_{k}$)be the pro-unipotent radical of $I$ (resp. $I_{k}$). We note that the inclusion $\mS\to \mI$ induces an isomorphism $\mS_\kappa\cong I/I^+$. The group $I/I^+$ was denoted by $\mS^{\natural}_\kappa$ in the introduction, but from now on we simply denote it by $\mS_\kappa$.

\subsection{Relating mixed and equal characteristics}\label{ss:mix eq}
Next we recall the passage from $G$ to $G^\flat$ as in \cite[\S 3]{PZ17}. Recall $F^{\flat}=\kappa(\!(u)\!)$.
Consider the ring $\mO[u,u^{-1}]$. There are two specialization maps 
\begin{equation}\label{eq: two sp map}
\mO[u,u^{-1}]\to F, \ u\mapsto \varpi,\quad\quad \mO[u,u^{-1}]\to F^\flat, \ \mO\to \kappa, u\mapsto u
\end{equation}
which induce equivalences between the following categories: (1) finite \'etale extension of $\mO[u^{\pm 1}]$; (2) finite \'etale extension of $F$ that splits over a tamely ramified extension of $F$; and (3) finite \'etale extension of $F^{\flat}$ %\kappa(\!(u)\!)$ 
that splits over a tamely ramified extension.
Thus the $\Xi_H$-torsor $\xi_G$ in \eqref{eq: xiG} can be (uniquely) extended to a $\Xi_H$-torsor $\underline \xi_G$ on $\mO[u^{\pm 1}]$. We let
\[
(\underline G,\underline B,\underline T,\underline e)=(H, B_H, T_H, e_H)\times^{\Xi_H}\underline \xi_G.
\]
Then $\un{G}$ is a connected reductive group over $\mO[u,u^{-1}]$ equipped with a pinning $(\underline B, \underline T, \underline e)$, together with an isomorphism
\[
(\underline G, \underline B, \underline T, \underline e)\otimes_{\mO[u^{\pm 1}],u\mapsto \varpi} F\cong (G,B,T,e).
\] 
There are subtori $\underline A\subset \underline S$ of $\underline T$, where $\underline A$ is split and $\underline S$ splits over $\breve\mO[u^{\pm 1}]$ such that $(\underline A,\underline S)\otimes_{\mO[u^{\pm 1}]} F=(A,S)$. 

We let 
\[
(G^\flat,B^\flat, T^\flat, e^\flat)=(\underline G, \underline B, \underline T, \underline e)\otimes_{\mO[u^{\pm1}]}F^\flat,
\] 
which is a connected reductive group over $F^\flat$ equipped with a pinning. We also have the corresponding subtori $A^\flat\subset S^\flat\subset T^\flat$ via base change of $\un A\subset \un S\subset \un T$, and the Iwahori Weyl group $\widetilde{W}^\flat$.

Let $\br F^{\flat}=k(\!(u)\!)$ with valuation ring $\br \mO^{\flat}=k[[u]]$. Let $\sigma^{\flat}\in \Aut(\br F^{\flat}/F^{\flat})$ be the $\sharp\kappa$-Frobenius. By \cite[\S 4]{PZ17}, we have a canonical identification of the apartment
\begin{equation*}\label{eq:identifying apartment}
\scrA(G, S, \breve F)=\scrA(G^\flat, S^\flat, \br F^{\flat})
\end{equation*}
as simplicial complexes. Indeed, both are identified with $\scrA(H,T_H,\tilde F\breve F)^{\Gal(\tilde F\breve F/\breve F)}$. We shall simply denote them as $\scrA$. 
In addition, this identification is compatible with the action of Iwahori-Weyl group $\widetilde W=\widetilde W^\flat$ and the action of the Frobenius $\sigma$ and $\sigma^\flat$. Below we denote $\wt W^{\flat}$ and $\sigma^{\flat}$ simply by $\wt W$ and $\sigma$.

We denote by $\br\mP^{\flat}_{\br\bff}, \br\mI^{\flat}, \mI^{\flat}, I^{\flat}_{k} ,I^{\flat}, I^{\flat,+}_{k}$ and $I^{\flat, +}$ the counterparts of $\br\mP_{\br\bff}, \br\mI,\mI, I_{k}, I, I^{+}_{k}$ and $I^{+}$ for $G^{\flat}$.

By \cite[Theorem 4.1]{PZ17}, for every facet $\breve\bff\subset \scrA$, there is a unique smooth affine group scheme $\underline{\breve\mP}_{\breve\bff}$ over $\breve\mO[u]$ such that the specializations
\[
\breve\mP_{\breve\bff}:=\underline{\breve\mP}_{\breve\bff}\otimes_{\breve\mO[u], u\mapsto \varpi}\breve\mO,\quad\quad \breve\mP^\flat_{\breve\bff}:=\underline{\breve\mP}_{\breve\bff}\otimes_{\br\mO[u]}\br\mO^{\flat}
\]
are the parahoric group schemes over $\breve\mO$ and $\br\mO^{\flat}$ corresponding to $\br\bff$. If in addition $\breve\bff$ is $\sigma$-stable, then $\underline{\breve\mP}_{\breve\bff}$ descends to $\mO[u]$ and specializes to the corresponding parahoric group scheme of $G$ (resp. $G^\flat$) over $\mO$ (resp. $\mO^{\flat}$). 

We will need the following lemma.
\begin{lem}\label{lem: map between parahoric}
If $\breve\bff$ is contained in the closure of $\breve\bff'$, then there is a unique morphism $\underline{\breve\mP}_{\breve\bff'}\to \underline{\breve\mP}_{\breve\bff}$ of group schemes over $\br\mO[u]$ that restricts to the identity map over $\breve\mO[u^{\pm1}]$ and specializes to the corresponding maps between parahoric group schemes for $G$ and $G^{\flat}$ respectively.
\end{lem}
\begin{proof} We need the following property of the group scheme $\underline{\breve\mP}_{\breve \bff}$ established in \cite[\S 4]{PZ17}. The apartment $\scrA(G, S, \breve F(\!(u)\!))$ is also canonically identified with $\scrA$. Here we regard $\breve F(\!(u)\!)$ as a local field with residue field $\breve F$. Then $\underline{\breve\mP}_{\breve\bff}\otimes_{\breve\mO[u]}\breve F[[u]]$ is the parahoric group scheme of $G_{\breve F(\!(u)\!)}$ associated to $\breve\bff$. It follows from the classical Bruhat-Tits theory that the identity map of $\underline G$ uniquely extends to a homomorphism $\iota: \underline{\breve\mP}_{\breve\bff'}|_{\Spec \breve\mO[u]\setminus\{s\}}\to \underline{\breve\mP}_{\breve\bff}|_{\Spec \breve\mO[u]\setminus\{s\}}$, where $s$ is the closed point of $\Spec \breve\mO[u]$ defined by the maximal ideal $(\varpi, u)$. Using the fact that  $\underline{\breve\mP}_{\breve \bff'}$ and  $\underline{\breve\mP}_{\breve \bff}$ are affine smooth, $\iota$ extends to uniquely over $\Spec \mO[u]$ as desired. 
\end{proof}

Let $L_p^+\breve\mP_{\breve\bff}$ be the $p$-adic positive loop group of $\breve\mP_{\breve\bff}$ and $L^+\breve\mP_{\breve\bff}^\flat$ the equal characteristic positive loop group. They are pro-algebraic groups over $k$. We have canonical maps
\[
L_p^+\breve\mP_{\breve\bff}\twoheadrightarrow \underline{\breve\mP}_{\breve\bff}\otimes_{\breve\mO[u],u=\varpi=0}k\twoheadleftarrow L^+\breve\mP_{\breve\bff}^\flat.
\]
It follows that the reductive quotients of both $L_p^+\breve\mP_{\breve\bff}$ and $L^+\breve\mP_{\breve\bff}^\flat$ are identified with $M_{\breve\bff}$. If in addition $\breve\bff$ is $\sigma$-stable, the above diagram descends to $\kappa$.

Note that if $\breve\bff$ is contained in the closure of $\breve\bff'$, then by Lemma \ref{lem: map between parahoric} there is the following commutative diagram which descends to $\kappa$ if $\breve\bff$ is $\sigma$-stable
\begin{equation*}\label{eq: map par mixed vs equal}
\xymatrix{
L_p^+\breve\mP_{\breve\bff'}\ar[r]\ar[d] & M_{\breve\bff'}\ar[d] & \ar[l] L^+\breve\mP_{\breve\bff'}^\flat\ar[d]\\
L_p^+\breve\mP_{\breve\bff}\ar[r] & M_{\breve\bff} & \ar[l] L^+\breve\mP_{\breve\bff}^\flat.
}
\end{equation*}
The right and the middle vertical maps of the above diagram are closed embeddings, while the left one becomes a closed embedding after perfection. As our main objects to study are $\ell$-adic sheaves on these spaces, we pass to perfection in the sequel. However, by abuse of notations, we shall still use $L^+\breve\mP_{\breve\bff}^\flat$ and $M_{\breve\bff}$ to denote their perfection. We will also use $L^+\breve\mP_{\breve\bff}$ to denote the perfection of $L^+_{p}\breve\mP_{\breve\bff}$.

In particular, if $\breve\bff$ is in the closure of the chosen alcove $\breve\bfa$, then $I_k\subset L^+\breve\mP_{\breve\bff}$ and $I^\flat_k\subset L^+\breve\mP^\flat_{\breve\bff}$. Both $I_k$ and $I^\flat_k$ have the same image $B_{\breve\bff}$ in $M_{\breve\bff}$, which is a Borel subgroup. Similarly, both $I^+_k$ and $I^{\flat,+}_k$ have the same image $U_{\breve\bff}$ in $M_{\breve\bff}$, which is the unipotent radical of $B_{\breve\bff}$. We have the identification of tori $I_k/I^+_k=B_{\breve\bff}/U_{\breve\bff}=I^{\flat}_k/I^{\flat,+}_k$.

\subsection{Bruhat-Tits group schemes}
For later purpose, we need  Bruhat-Tits type group schemes over $\breve\mO[u]$ containing $\underline{\breve\mP}_{\breve\bff}$ as open subgroup schemes. In fact, we just need a very special case when $\breve\bff=\breve\bfa$. But we will present a general construction, which might be of independent interest.

We let $\scrB^{\ext}(G,\breve F)=\scrB(G,\breve F)\times \xcoch(S_\ab)_\bR$ be the extended Bruhat-Tits building, and identify $\scrB(G,\breve F)$ with $\scrB(G,\breve F)\times \{0\}\subset \scrB^{\ext}(G,\breve F)$.

For $x\in \scrA(G,S,\breve F)$, let $\breve\mG_x$ denote the corresponding {\em Bruhat-Tits group scheme} of $x$: it is the unique smooth group scheme over $\breve\mO$ such that $\breve\mG_x(\breve\mO)$ is the stabilizer of $x$ under the action of $G(\breve F)$.
Then $\breve\mG_x^{\circ}=\breve\mP_{\breve\bff}$ where $\breve\bff$ is the facet containing $x$. If $x$ is in the closure of the alcove $\breve\bfa$, then we have a canonical isomorphism
\begin{equation}\label{eq:omegax}
\pi_0(\breve\mG_x\otimes_{\breve\mO} k)\cong\Omega_x:=\{\omega\in\Omega\mid \omega x=x\}.
\end{equation}
If $\pi_0(\breve\mG_x)$ is finite, then $\breve\mG_x$ is affine.
We shall write $M_x$ for the quotient of $\breve\mG_x\otimes_{\breve\mO}k$ by the unipotent radical of its neutral connected component, which is not necessarily connected and its neutral component is $M_{\breve\bff}$ for the facet $\br\bff$ containing $x$. If $x$ is fixed by $\sigma$, then $\breve\mG_x$ descends to $\mO$ and $M_x$ descends to $\kappa$.

Similarly, for $G^{\flat}$, we have the  Bruhat-Tits group schemes $\br\mG^{\flat}_{x}$ over $\br\mO^{\flat}$, and the reductive quotient $M^{\flat}_{x}$ of its special fiber.

We will need constructions parallel to \S\ref{ss:mix eq} for Bruhat-Tits group schemes. %We assume that $G$ is semisimple. 
As argued in \cite[Proposition 3.3]{DLYZ25}, there exists a set-theoretic map 
\[
\widetilde{W}\to N_{\underline G}(\underline S)(\breve\mO[u^{\pm 1}])
\]
such that its two specializations in \eqref{eq: two sp map} give set-theoretic sections of the natural maps $N_G(S)(\breve F)\to \widetilde W$ and $N_{G^\flat}(S^\flat)(\br F^{\flat})\to \widetilde W$. In particular, we can fix a set of representatives $\{v_\omega\}_{\omega\in\Omega}\subset N_{\underline G}(\underline S)(\breve\mO[u^{\pm 1}])$  of $\Omega$, whose two specializations  give representatives of $\Omega$ in $N_G(S)(\breve F)$ and in $N_{G^\flat}(S^\flat)(\br F^{\flat})$ respectively. We will always assume that $v_1=1$.

Now we construct the Bruhat-Tits group scheme over $\breve\mO[u]$, partially generalizing \cite[Theorem 4.1]{PZ17}.

\begin{prop}
Let $x\in \overline{\breve\bfa}$ be a point. Then there is a unique smooth group scheme $\underline {\breve \mG}_x$ over $\breve\mO[u]$ extending $\underline G$, such that its two specializations as in \eqref{eq: two sp map} give the corresponding Bruhat-Tits group schemes $\breve\mG_x$ and $\breve\mG_x^\flat$ respectively. If $\Omega_x$ is finite, then $\underline{\breve\mG}_x$ is affine.
\end{prop}
Of course, once we have the construction of $\underline {\breve \mG}_x$ for $x\in \overline{\breve\bfa}$, we have the construction for all points in $\scrB^{\ext}(G,\breve F)$.
\begin{proof}
The argument is the same as \cite[Proposition 4.6.18]{BT84}. 
Without loss of generality, we may assume that $x$ is in the closure of the alcove $\breve\bfa$. Recall the subgroup $\Omega_x\subset\Omega$ from \eqref{eq:omegax}. For $\omega,\omega'\in\Omega_x$, let
\begin{equation*}\label{fom}
f_{\omega\omega'}: \underline G_{\breve\mO[u^{\pm}]}\to  \underline G_{\breve\mO[u^{\pm}]},\quad g\mapsto gv_{\omega'}^{-1}v_{\omega}.
\end{equation*}
As $f_{\omega\omega'}\circ  f_{\omega'\omega''}=f_{\omega\omega''}$, we can use this system to glue $\sqcup_{\Omega_x}\underline {\br\mP}_{\breve\bff}$ to obtain a smooth scheme $\underline {\breve\mG}_x$ over $\breve\mO[u]$. It is clearly quasi-separated. To show it is affine when $\Omega_x$ is finite, we notice that under this assumption $\underline{\mG}_x$ is quasi-compact. Then it is enough to show that $H^i(\underline{\breve\mG}_x,\mF)=0$ for every quasi-coherent sheaf $\mF$ on $\underline{\breve\mG}_x$ and every $i>0$. But this follows from an easy \v{C}ech cohomology computation using the natural covering of $\underline{\breve\mG}_x$ by affine open subsets $\sqcup_{\Omega_x}\underline \mP_{\breve\bff}$.

Using Lemma \ref{lem: map between parahoric}, the same argument as in \cite[Proposition 4.6.18]{BT84} shows that $\underline{\breve\mG}_x$ has a natural group scheme structure extending that of $\underline{\breve\mP}_{\breve\bff}$.
Its two specializations give the corresponding group schemes $\breve\mG_x$ over $\breve\mO$ and $\breve\mG_x^\flat$ over $\breve\mO^{\flat}$ respectively, by the construction of $\mG_x$ and $\mG_x^\flat$ as in \cite[Proposition 4.6.18]{BT84}. Finally, the argument as in \cite[4.2.1]{PZ17} shows that $\underline{\breve\mG}_x|_{\breve F[[u]]}$ is also the Bruhat-Tits group scheme of $G_{\breve F(\!(u)\!)}$ corresponding to $x\in \scrA(G, S, \breve F(\!(u)\!))$, and therefore is unique.
\end{proof}

\subsection{Diagrams of Hecke stacks: equivariant version}\label{ss:Hk diag eq}
We will compute affine Hecke categories as the colimits of finite Hecke categories indexed by certain diagrams. In the next two subsections, we introduce and compare these diagrams in mixed and equal characteristic settings. This subsection deals with the unipotent case and the next subsection deals with the monodromic case. Strictly speaking, we will only need the monodromic case. However, we feel that it might be technically slightly easier if we first present the unipotent case and indicate the needed modifications in the monodromic case.

We consider the following partially ordered set $\frakP ar$ whose elements are facets $\breve\bff$ contained in the closure of the alcove $\breve\bfa$, and whose partial order is given by requiring $\breve\bff'\leq \breve\bff$ if $\breve\bff$ is contained in the closure of $\breve\bff'$. We consider $\frakP ar$ as an ordinary category. 
Consider three functors
\begin{equation}\label{eq: natural transform for classifying stack of parahorics}
\bB L^+\breve\mP,\ \bB L^+\breve\mP^\flat,\ \bB M: \frakP ar\to \prestk_k,
\end{equation}
where $\prestk_k$ denotes the category of perfect prestacks over $k$ (see \cite[\S 10.1]{Zhu25} for details). The first sends $\breve\bff$ to $\bB L^+\breve\mP_{\breve\bff}$, the second sends $\breve\bff$ to $\bB L^+\breve\mP^\flat_{\breve\bff}$ and the third sends $\breve\bff$ to $\bB M_{\breve\bff}$. Here, as usual for an affine group scheme $H$ over $k$, $\bB H$ denotes the classifying stack of $H$ in the \'etale topology.  The previous commutative diagram \eqref{eq: map par mixed vs equal} induces natural transformations 
\begin{equation}\label{eq: natural transform for classifying stack of parahorics-2}
\breve\bff\mapsto (\bB L^+\breve\mP_{\breve\bff}\to \bB M_{\breve\bff}\leftarrow  \bB L^+\breve\mP_{\breve\bff}^\flat)
\end{equation}
between these functors.

Let $\corr(\prestk_k)$ be the symmetric monoidal category of correspondences (e.g. see \cite[\S 8.1]{Zhu25} for a summary). 
As explained in \emph{loc. cit.}, taking the \v{C}ech nerves of $\bB I_k\to \bB L^+\breve\mP_{\breve\bff}$ produces a Segal object in $\corr(\prestk_k)$, which in turn endows $I_k\bs L^+\breve\mP_{\breve\bff}/I_k=\bB I_k\bs \bB L^+\breve\mP_{\breve\bff}/ \bB I_k$ with an algebra structure in  $\corr(\prestk_k)$. That is,  $I_k\bs L^+\breve\mP_{\breve\bff}/I_k$ is canonically lifted to an object in the category $\Alg(\corr(\prestk_k))$. If we regard $\bB I_k$ as a constant functor from $\frak P ar$ to $\prestk_k$ and apply this construction to $\bB I_k\to \bB L^+\breve\mP$, where $\bB L^+\breve\mP$ denotes the functor from \eqref{eq: natural transform for classifying stack of parahorics}, then we obtain a functor 
\[
\mH k^\circ:  \frakP ar\to \Alg(\corr(\prestk_k)),\quad \breve\bff\to I_k\bs L^+\breve\mP_{\breve\bff}/I_k.
\] 
The above construction applies verbatim to the functor $\bB L^+\breve\mP_{\breve\bff}^\flat$ to obtain
\[
\mH k^{\flat,\circ}:  \frakP ar\to \Alg(\corr(\prestk_k)),\quad \breve\bff\to I^{\flat}_k\bs L^+\breve\mP^{\flat}_{\breve\bff}/I^{\flat}_k.
\]
Using the Borel subgroup $B_{\br\bff}$ of $M_{\br\bff}$, we can apply the same construction to the functor $\bB M$ to obtain a functor
\[
\ov\mH k^{\flat,\circ}:  \frakP ar\to \Alg(\corr(\prestk_k)),\quad \breve\bff\to  B_{\br\bff}\bs M_{\breve\bff}/B_{\br\bff}.
\]
These functors are related by natural transformations
\begin{equation}\label{eq: natural transform for Hecke stack}
\mH k^\circ\to \overline{\mH k}^\circ\leftarrow \mH k^{\flat,\circ}: \frakP ar\to \Alg(\corr(\prestk_k)).
\end{equation}

To continue, recall that $\Omega$ acts on $\breve\bfa$ and therefore acts on $\frakP ar$. On the other hand, there is a canonical isomorphism $N_{LG_k}(I_k)/I_k\cong \Omega$. Therefore, for $\omega\in\Omega$, conjugation by a lifting $\tilde\omega\in N_{LG_k}(I_k)$ induces an isomorphism in $\Alg(\prestk_k)$ 
\[
I_k\bs L^+\breve\mP_{\breve\bff}/I_k \cong I_k\bs L^+\breve\mP_{\omega(\breve\bff)}/I_k,
\] 
which is canonically independent of the choice of the lifting. Similarly, we have a canonical isomorphism $N_{LG_k^\flat}(I_k^\flat)/I_k^\flat\cong \Omega$ and therefore a canonical isomorphism $I^\flat_k\bs L^+\breve\mP^\flat_{\breve\bff}/I^\flat_k \cong I^\flat_k\bs L^+\breve\mP^\flat_{\omega(\breve\bff)}/I^\flat_k$ for $\omega\in \Omega$. In addition, there is a canonical isomorphism $B_{\breve\bff}\bs M_{\breve\bff}/B_{\breve\bff}\to B_{\omega(\breve\bff)}\bs M_{\omega(\breve\bff)}/B_{\omega(\breve\bff)}$ fitting into the following commutative diagram
\begin{equation}\label{omega diag}
\xymatrix{
I_k\bs L^+\breve\mP_{\breve\bff}/I_k\ar[r]\ar[d] & B_{\breve\bff}\bs M_{\breve\bff}/B_{\breve\bff} \ar[d]& \ar[l]I^\flat_k\bs L^+\breve\mP^\flat_{\breve\bff}/I^\flat_k\ar[d] \\
I_k\bs L^+\breve\mP_{\omega(\breve\bff)}/I_k\ar[r] & B_{\omega(\breve\bff)}\bs M_{\omega(\breve\bff)}/B_{\omega(\breve\bff)} &\ar[l]I^\flat_k\bs L^+\breve\mP^\flat_{\omega(\breve\bff)}/I^\flat_k.
}
\end{equation}
This follows from the fact the conjugation action of $v_\omega$ on $\underline{G}$ extends to an isomorphism $\underline{\breve\mP}_{\breve\bff} \to \underline{\breve\mP}_{\omega(\breve\bff)}$ by Lemma \ref{lem: map between parahoric}.

Similarly, the Frobenius $\sigma\in \Gal(\breve F/F)=\Gal(\breve F^\flat/F^\flat)$ acts on everything and we have a similar commutative diagram as \eqref{omega diag} where $\omega$ is replaced by $\sigma$.
Formally, we can summarize to say the functors in \eqref{eq: natural transform for Hecke stack} descend to functors
\[
\mH k^{\circ,\Omega\rtimes\langle\sigma\rangle}\to \overline{\mH k}^{\circ,\Omega\rtimes\langle\sigma\rangle}\leftarrow \mH k^{\flat,\circ,\Omega\rtimes\langle\sigma\rangle}:  \frakP ar/(\Omega\rtimes\langle\sigma\rangle)\to \Alg(\corr(\prestk_k)).
\]
Here $\frakP ar/(\Omega\rtimes\langle\sigma\rangle)$ denotes the quotient groupoid of $\frakP ar$ by $\Omega\rtimes\langle\sigma\rangle$, which is the geometric realization (in the $\infty$-category of $\infty$-categories) of the simplicial category $(\Omega\rtimes\langle\sigma\rangle)^\bullet\times\frakP ar$.

Next, let $\frakP ar^{\triangleright}$ be the right cone of $\frakP ar$. The two functors $\bB L^+\breve\mP$ and $\bB L^+\breve\mP^\flat$ in \eqref{eq: natural transform for classifying stack of parahorics} natural extends to functors $\frakP ar^{\triangleright}\to \prestk_k$ by sending the cone point to $\bB LG_k^\circ$ and $\bB LG_k^{\flat,\circ}$ respectively. Therefore, the two functors $\mH k^\circ, \mH k^{\flat,\circ}$ in \eqref{eq: natural transform for Hecke stack} extend to $\frakP ar^{\triangleright}$ sending the cone point to $I_k\bs LG_k^\circ/I_k$ and $I_k^\flat\bs LG_k^{\flat,\circ}/I_k^{\flat}$ respectively. Using the conjugation action of $\Omega$ on $I_k\bs LG_k^\circ/I_k$ and $I_k^\flat\bs LG_k^{\flat,\circ}/I_k^{\flat}$, these functors further descend to the quotient $\frakP ar^{\triangleright}/(\Omega\rtimes\langle\sigma\rangle)$
\[
\mH k^{\circ,\Omega\rtimes\langle\sigma\rangle,\triangleright}, \mH k^{\flat, \circ, \Omega\rtimes\langle\sigma\rangle,\triangleright}: \frakP ar^{\triangleright}/(\Omega\rtimes\langle\sigma\rangle)\to \Alg(\corr(\prestk_k)).
\]

\subsection{Diagram of Hecke categories: monodromic version}
There are also monodromic analogues of \eqref{eq: natural transform for Hecke stack}. To explain its meaning, let $\bfC$ be the category of pairs $(H,X)$ consisting of a perfect prestack over $k$ equipped with an action of a torus $H$. This category admits finite product. See the discussion before \cite[Proposition 4.28]{Zhu25}.
We may endow each space in \eqref{eq: natural transform for classifying stack of parahorics-2} with the trivial action of the trivial torus, and regard functors in \eqref{eq: natural transform for classifying stack of parahorics} from $\frakP ar$ to $\bfC$. We equip $\bB I^+_k$ with the natural action of $\mS_k$. Then $I^+_k\bs L^+\breve\mP_{\breve\bff}/I^+_k$ is equipped with the $(\mS_k\times\mS_k)$-action by left and right translations, and can be naturally regarded as an object in $\Alg(\corr(\bfC))$. Repeating the previous construction then gives $\mH k^{\circ,+}: \frakP ar\to \Alg(\corr(\bfC))$ sending $\breve\bff$ to  $(\mS_{k}\times \mS_{k},I^+_k\bs L^+\breve\mP_{\breve\bff}/I^+_k)$. 
Similar discussions apply to the equal characteristic setting, and the finite setting, giving 
\begin{equation}\label{eq: natural transform for mono Hecke stack}
\mH k^{\circ,+}\to \overline{\mH k}^{\circ,+}\leftarrow \mH k^{\flat,\circ,+}: \frakP ar\to \Alg(\corr(\bfC)).
\end{equation}
To obtain the equivariant structures on these functors, we need one new ingredient. Let $C$ be the barycenter of $\breve\bfa$ and consider the Bruhat-Tits group scheme $\underline{\breve\mG}_{C}$ over $\breve\mO[u]$, and its specializations $\breve\mG_C$ over $\br\mO$ and $\br\mG_{C}^{\flat}$ over $\br\mO^{\flat}$. We have canonical identifications $\pi_{0}(\breve\mG_C)=\pi_{0}(\br\mG_{C}^{\flat})=\Omega$.

Let
\[
L:=\breve\mG_C\otimes_{\breve\mO}k=\underline{\breve\mG}_C\otimes_{\breve\mO[u], u=\varpi=0}k=\breve\mG_C^{\flat}\otimes_{k[[u]]}k.
\]
It is equipped with the following filtration by closed normal subgroups
\[
L=L_{-1}\triangleright L_0\triangleright L_1,
\] 
where $L_0=L^\circ$ is the neutral component, $L_1$ is the unipotent radical of $L_0$. Then $L/L_1$ is an extension of $\Omega$ by $\mS_k=\mS^\flat_k$. We let
\begin{equation*}
    \wt\Omega:=(L/L_1)(k)\supset \wt\Omega^\circ:=(L_0/L_1)(k)=\mS_k(k)=\mS^\flat_k(k),
\end{equation*}
both as abstract groups. We have $\wt\Omega/\wt\Omega^\circ\cong\Omega$. 
Since $L/L_1$ is in fact defined over $\kappa$, the Frobenius $\sigma$ acts on $\wt\Omega$ by automorphisms, preserving the subgroup $\wt\Omega^\circ$. 

We note that there are canonical isomorphisms
\[
(N_{LG_k}(I^+_k)/I^+_k)(k)\cong \wt\Omega\cong (N_{LG^\flat_k}(I^{\flat,+}_k)/I^{\flat,+}_k)(k),
\]
compatible with the $\sigma$-action. We thus can repeat the discussion of \S\ref{ss:Hk diag eq}, with the role of $\Omega\rtimes \langle\sigma\rangle$ played by $\wt\Omega\rtimes \langle\sigma\rangle$, to descend $\mH k^{\circ,+}, \overline{\mH k}^{\circ,+}$ and $\mH k^{\flat,\circ,+}$ to
\[
\mH k^{\circ,+,\wt\Omega\rtimes\langle\sigma\rangle}\to \overline{\mH k}^{\circ,+,\wt\Omega\rtimes\langle\sigma\rangle}\leftarrow \mH k^{\flat,\circ,+,\wt\Omega\rtimes\langle\sigma\rangle}:  \frakP ar/(\wt\Omega\rtimes\langle\sigma\rangle)\to \Alg(\corr(\bfC)).
\]
In addition, we have extensions
\[
\mH k^{\circ,+,\wt\Omega\rtimes\langle\sigma\rangle,\triangleright}, \mH k^{\flat, \circ,+,\wt\Omega\rtimes\langle\sigma\rangle,\triangleright}: \frakP ar^{\triangleright}/(\wt\Omega\rtimes\langle\sigma\rangle)\to \Alg(\corr(\bfC))
\]
sending the cone point to $(\mS_{k}\times \mS_{k}, I^{+}_{k}\bs LG_{k}^{\circ}/I^{+}_{k})$ and $(\mS_{k}\times \mS_{k}, I^{\flat,+}_{k}\bs LG^{\flat,\circ}_{k}/I^{\flat,+}_{k})$.

We need one more piece of structure. We regard $\wt\Omega$ as a discrete group scheme over $k$, equipped with the action of the trivial torus. Then $\wt\Omega$ can be regarded as an object in $\Alg(\corr(\bfC))$. There is a natural homomorphism in $\Alg(\corr(\bfC))$,
\begin{equation}\label{eq:wtOmega to Hk}
\wt\Omega\to I^+_k\bs LG_k/I^+_k,\quad \tilde\omega\mapsto I^+_k\bs I^+_k\tilde\omega/I^+_k,
\end{equation}
which restricts to a homomorphism $\wt\Omega^\circ\to I^+_k\bs LG^\circ_k/I^+_k$. In addition, this map factors through $\wt\Omega^\circ\to I^+_k\bs L^+\breve\mP_{\breve\bff}/I^+_k\subset I^+_k\bs LG^\circ_k/I^+_k$ for each $\breve\bff$. Clearly, there is the corresponding equal characteristic version. Finally, the two maps $\wt\Omega^\circ\to I^+_k\bs L^+\breve\mP_{\breve\bff}/I^+_k\to U_{\breve\bff}\bs M_{\breve\bff}/U_{\breve\bff}$ and $\wt\Omega^\circ\to I^{\flat,+}_k\bs L^+\breve\mP^\flat_{\breve\bff}/I^{\flat,+}_k\to U_{\breve\bff}\bs M_{\breve\bff}/U_{\breve\bff}$ are canonically isomorphic.

We formally organize this structure as follows.
Consider the constant functor $\frakP ar^{\triangleright}\to \Alg(\corr(\bfC))$ sending $\breve\bff$ to $\wt\Omega^{\circ}$, which descends to $\frakP ar^{\triangleright}/(\wt\Omega\rtimes\langle\sigma\rangle) \to \Alg(\corr(\bfC))$ using the action of $\wt\Omega\rtimes\langle\sigma\rangle$ on $\wt\Omega^\circ$. We shall denote this functor by $\wt\Omega^{\circ,\wt\Omega\rtimes\langle\sigma\rangle,\triangleright}$. Then we have the following commutative diagram of natural transformation of functors
\begin{equation}\label{eq: all in one diagram}
\xymatrix{
\mH k^{\circ,+,\wt\Omega\rtimes\langle\sigma\rangle,\triangleright}\ar@{..>}[d]&\ar[l]\wt\Omega^{\circ,\wt\Omega\rtimes\langle\sigma\rangle,\triangleright}\ar[r]\ar@{..>}[d]& \mH k^{\flat,\circ,+,\wt\Omega\rtimes\langle\sigma\rangle,\triangleright}\ar@{..>}[d]\\
\mH k^{\circ,+,\wt\Omega\rtimes\langle\sigma\rangle}\ar[dr]&\ar[l]\wt\Omega^{\circ,\wt\Omega\rtimes\langle\sigma\rangle}\ar[r]& \mH k^{\flat,\circ,+,\wt\Omega\rtimes\langle\sigma\rangle}\ar[dl]\\
&\overline{\mH k}^{\circ,+,\wt\Omega\rtimes\langle\sigma\rangle}.&}
\end{equation}
Here dotted arrow means restriction of functors along $\frakP ar/(\wt\Omega\rtimes\langle\sigma\rangle)\to \frakP ar^{\triangleright}/(\wt\Omega\rtimes\langle\sigma\rangle)$.

\section{Proof of the theorem}

\subsection{Neutral block}

We first prove a weaker version of Theorem \ref{intro: main thm}: the full subcategories of the affine Hecke categories consisting of sheaves supported on the neutral component of the loop group
\begin{eqnarray*}    \mH^\circ_{\mon,\af}:=\Shv\bigl((I_k,\mon)\bs LG^\circ_k/(I_k,\mon)\bigr), &\quad &\mH^{\flat,\circ}_{\mon,\af}:=\Shv\bigl((I_k,\mon)\bs LG^{\flat,\circ}_k/(I_k,\mon)\bigr),\\
\mH^{\circ}_{\unip,\af}=\Shv(I_k\bs LG^\circ_k/I_k), &\quad &\mH^{\flat,\circ}_{\unip,\af}=\Shv(I^\flat_k\bs LG^{\flat,\circ}_k/I^\flat_k),
\end{eqnarray*}
are canonically equivalent in each line.

For a perfect prestack $X$ over $k$, let $\Shv(X)$ denote the $\infty$-category of $\ell$-adic (co)sheaves on $X$ with coefficient in $\La=\bZ_\ell$. (See \cite[\S 10.4]{Zhu25} for details.) Let $\lincat_\La$ denote the category of $\La$-linear presentable stable $\infty$-categories, equipped with Lurie's tensor product.  
There is a lax symmetric monoidal functor
\[
\Shv: \corr(\prestk_k)_{\mathrm{V};\mathrm{All}}\to \lincat_\La, \quad X\mapsto \Shv(X)
\]
as constructed in \cite[(10.47)]{Zhu25}, where $\mathrm{V}$ is a class of morphisms in $\prestk_k$ containing perfectly finitely presented morphisms.  We note that $\bB I_k\to \bB L^+\breve\mP_{\breve\bff}$ is perfectly finitely presented and proper, and similarly in other two situations. Therefore,  $\mH k^\circ, \mH k^{\flat,\circ}, \overline{\mH k}^{\circ}$ take value in $\Alg(\corr(\prestk_k)_{\mathrm{V};\mathrm{All}})$. 

Now composing $\Shv$ with functors appearing in \eqref{eq: natural transform for Hecke stack}, we obtain three functors $\frakP ar\to \Alg(\lincat_\La)$ and natural transformations between them, denoted as 
\begin{equation}\label{eq: colimit unipotent Hecke category mix and eq}
\mH^{\circ}_{\unip}\to \overline{\mH}^{\circ}_{\unip}\leftarrow \mH_{\unip}^{\flat,\circ},
\end{equation}
where 
\begin{itemize}
\item $\mH^\circ_{\unip}$ sends $\breve\bff$ to $\Shv\bigl(I_k\bs L^+\breve\mP_{\breve\bff}/I_k\bigr)$;
\item $\mH_{\unip}^{\flat,\circ}$ sends $\breve\bff$ to $\Shv\bigl(I^\flat_k\bs L^+\breve\mP^\flat_{\breve\bff}/I^\flat_k\bigr)$; 
\item $\overline{\mH}^\circ_{\unip}$ sends $\breve\bff$ to $\Shv\bigl(B_{\breve\bff}\bs M_{\breve\bff}/ B_{\breve\bff}\bigr)$. 
\end{itemize}
We also notice that all functors in \eqref{eq: colimit unipotent Hecke category mix and eq} descend to 
\[
\mH^{\circ,\Omega\rtimes\langle\sigma\rangle}_{\unip}\to \overline{\mH}^{\circ,\Omega\rtimes\langle\sigma\rangle}_{\unip}\leftarrow \mH_{\unip}^{\flat,\circ,\Omega\rtimes\langle\sigma\rangle}: \frakP ar/(\Omega\rtimes\langle\sigma\rangle)\to \Alg(\corr(\prestk_k)_{\mathrm{V};\mathrm{All}}),
\] 
and $\mH^{\circ,\Omega\rtimes\langle\sigma\rangle}_{\unip}, \mH_{\unip}^{\flat,\circ,\Omega\rtimes\langle\sigma\rangle}$ extend to 
\[
\mH^{\circ,\Omega\rtimes\langle\sigma\rangle,\triangleright}_{\unip}, \mH_{\unip}^{\flat,\circ,\Omega\rtimes\langle\sigma\rangle,\triangleright}:\frakP ar^{\triangleright}/(\Omega\rtimes\langle\sigma\rangle)\to \Alg(\corr(\prestk_k)_{\mathrm{V};\mathrm{All}}),
\]
where the cone points are sent $\mH^{\circ}_{\unip,\af}$ and $\mH^{\flat,\circ}_{\unip,\af}$, equipped with the $\Omega\rtimes\langle\sigma\rangle$-actions coming from the conjugation action of $\Omega$ on $I_k\bs LG^\circ_k/I_k$ and on $I^\flat_k\bs LG^{\flat,\circ}_k/I^\flat_k$, and the $*$-pushforward along the Frobenius endomorphisms.

Similarly we have a lax symmetric monoidal functor
\[
\Shv_\mon: \corr(\bfC)_{\mathrm{V};\mathrm{All}}\to \lincat_\La, \quad (H, X)\mapsto \Shv\bigl((H,\mon)\bs X\bigr),
\]
as constructed in \cite[Proposition 4.28]{Zhu25}. Then we can compose functors in \eqref{eq: natural transform for mono Hecke stack} with $\Shv_{\mon}$, giving 
\[
\mH^{\circ}_\mon\rightarrow \overline{\mH}^{\circ}_\mon\leftarrow \mH^{\flat,\circ}_\mon,
\]
and their various extensions.

We note that the map $I_k\bs L^+\breve\mP_{\breve\bff}/I_k\to B_{\breve\bff}\bs M_{\breve\bff}/ B_{\breve\bff}$ is a gerbe under the pro-unipotent group $\ker (I_k\to B_{\breve\bff})$ and similarly  $I^\flat_k\bs L^+\breve\mP^\flat_{\breve\bff}/I^\flat_k\to B_{\breve\bff}\bs M_{\breve\bff}/ B_{\breve\bff}$ is a gerbe under the pro-unipotent group $\ker (I^\flat_k\to B_{\breve\bff})$. Therefore all the natural transformations in \eqref{eq: colimit unipotent Hecke category mix and eq} are isomorphisms, matching (co)standard objects. Here for $w\in \widetilde{W}$ we normalize the corresponding (co)standard object to be the one whose restriction to the Schubert cell is the dualizing sheaf of the Schubert cell, shifted by $\ell(w)$. We can apply the same reasoning in the monodromic case. 

We summarize the above discussions into the following lemma.
\begin{lem}\label{lem: equiv of finite Hecke categories}
There are natural isomorphisms of functors
\[
\mH^{\circ,\Omega\rtimes\langle\sigma\rangle}_{\unip}\simeq \mH_{\unip}^{\flat,\circ,\Omega\rtimes\langle\sigma\rangle}: \frakP ar/(\Omega\rtimes\langle\sigma\rangle)\to \Alg(\lincat_\La),
\] 
which send (co)standard objects to (co)standard objects. Similarly, there are natural isomorphisms of functors sending (co)standard objects to (co)standard objects
\[
\quad \mH^{\circ,\wt\Omega\rtimes\langle\sigma\rangle}_{\mon}\simeq \mH_{\mon}^{\flat,\circ,\wt\Omega\rtimes\langle\sigma\rangle}: \frakP ar/(\wt\Omega\rtimes\langle\sigma\rangle)\to \Alg(\lincat_\La).
\]
\end{lem}

\begin{rem}
In the unipotent case, the above equivalences match (co)standard objects canonically. However,
In the monodromic case, there are cofree monodromic (co)standard objects as defined in \cite{DLYZ25, Zhu25}. They are a priori only defined up to isomorphism. The above equivalences only match them up to isomorphisms.
\end{rem}

Now we invoke the following important theorem of Tao and Travkin \cite{TT20}.
\begin{thm}[Tao-Travkin]\label{thm: TT sc}
The  functors 
\[
\mH^{\circ,\triangleright}_{\unip},\ \mH^{\flat, \circ,\triangleright}_{\unip},\ \mH^{\circ,\triangleright}_{\mon},\ \mH^{\flat, \circ,\triangleright}_{\mon}: \mathfrak Par^{\triangleright}\to \Alg(\lincat_\La)
\]
are colimit diagrams in $\Alg(\lincat_\La)$. 
\end{thm}

Concretely, denote by $\colim^{\Alg}$ the colimit taken in the category $\Alg(\lincat_\La)$. Then the functors induced by the inclusions $\mH_{\unip,\breve\bff}\to \Shv(I_k\bs LG^\circ_k/I_k)$ and $\mH^{\flat}_{\unip,\breve\bff}\to \Shv(I^\flat_k\bs LG^{\flat,\circ}_k/I^\flat_k)$
\[
\colim^{\Alg}_{\breve\bff\in\frakP ar}\mH_{\unip,\breve\bff}\to \mH_{\unip,\af}^\circ,\quad \colim^{\Alg}_{\breve\bff\in\frakP ar}\mH^{\flat}_{\unip,\breve\bff}\to \mH_{\unip,\af}^{\flat,\circ},
\] 
are equivalences of monoidal categories. Similarly the functors
\[
\colim^{\Alg}_{\breve\bff\in\frakP ar}\mH_{\mon,\breve\bff}\to  \mH_{\mon,\af}^\circ,\quad \colim^{\Alg}_{\breve\bff\in\frakP ar}\mH^{\flat}_{\mon,\breve\bff}\to \mH_{\mon,\af}^{\flat,\circ},
\] 
are equivalences of monoidal categories.

\begin{rem}
Tao-Travkin proved their theorems for unipotent affine Hecke category of simply-connected group in the D-module context. As explained in \emph{loc. cit.}, their argument is applicable in any six functor formalism satisfying  excision and $h$-descent. This is the case for both $\Shv$ and $\Shv_{\mon}$ in the \'etale setting. In addition, as noticed in \cite[Appendix C]{LNY24}, when $G$ is not simply-connected one can replace the whole unipotent affine Hecke category by its neutral block.
\end{rem}

Lemma~\ref{lem: equiv of finite Hecke categories}
and Theorem~\ref{thm: TT sc} together imply the following corollary.
\begin{cor}\label{cor: equiv netural block}
There is a canonical $\Omega\rtimes\langle\sigma\rangle$-equivariant (resp. $\wt\Omega\rtimes\langle\sigma\rangle$-equivariant) equivalence
\[
\mH_{\unip,\af}^\circ\cong \mH_{\unip,\af}^{\flat,\circ}, \quad (\mbox{resp. } \mH_{\mon,\af}^\circ\cong \mH_{\mon,\af}^{\flat,\circ})
\]
matching (co)standard objects.
\end{cor}
If $G$ is simply-connected, then each version of the affine Hecke category coincides with its neutral block. In particular,
Theorem~\ref{intro: main thm} is proved for simply-connected $G$.

\subsection{Contracted semidirect product}\label{ss:semidirect}
To deduce  the case of a general reductive group from the equivalence of neutral blocks, our strategy is to present $\mH_{\mon,\af}$ as a ``semidirect product" of its neutral block $\mH_{\mon,\af}^{\circ}$ and the component group $\Omega$ of $LG_k$. To carry this out, we need a formalism slightly more general than the categorical semidirect product (which was used in \cite{DLYZ25}) involving also a contraction.

Let us fix a group $\Ga$ and a normal subgroup $\Ga_0$. Let $\bfA_{\Ga,\Ga_0}$ be the $(2,1)$-category of $(\mM, \al, \be_0, \iota)$ where 
 \begin{itemize}
\item $\mM$ is a $\La$-linear ordinary additive monoidal category;
\item $\al: \Ga\to \Aut^{\otimes}(\mM)$ is a monoidal functor; \item $\be_0: \Ga_0\to \mM$ is a $\Ga$-equivariant monoidal functor; and
\item finally $\iota$ is an isomorphism of monoidal functors 
$$\iota: \Int\circ\be_0\simeq \al|_{\Ga_0}: \Ga_0\to \Aut^{\otimes}(\mM),$$ 
where $\Int\circ\be_0$ sends $\ga_0\in \Ga_0$ to the inner automorphism $\be_0(\ga_0)\otimes (-)\otimes \be_0(\ga_0^{-1})$ of $\mM$. We require $\iota$ be compatible with the $\Ga$-equivariant structures of both $\Int\circ\be_0$ and $\al|_{\Ga_0}$ ($\ga\in \Ga$ acts on $\Aut^{\otimes}(\mM)$ via conjugation by $\al(\ga)$).
\end{itemize}
    
Let $\bfB_{\Ga}$ be the category of pairs $(\mN, \be)$ where
\begin{itemize}
\item $\mN$ is a $\La$-linear ordinary additive monoidal category; and
\item  $\be: \Ga\to \mN$ is a monoidal functor. 
\end{itemize}

We have a forgetful functor
\[
r^{\Ga}_{\Ga_0}: \bfB_{\Ga}\to \bfA_{\Ga,\Ga_0}
\]
sending $(\mN, \be)$ to $(\mN, \Int\circ\be, \be|_{\Ga_0}, \iota)$ where $\iota$ is the canonical isomorphism between $\Int\circ(\be|_{\Ga_0})$ and $(\Int\circ\be)|_{\Ga_0}$.

Note that when $\Ga_0=\Ga$, $r^{\Ga}_\Ga$ gives an equivalence $\bfB_{\Ga}\cong\bfA_{\Ga,\Ga}$.

\begin{lem}\label{lem: adjunction}
    This forgetful functor admits a left adjoint 
    \[
    i^{\Ga}_{\Ga_0}:\bfA_{\Ga,\Ga_0}\to \bfB_{\Ga}.
    \]
    We denote the image of $(\mM, \al,\be_0, \iota)$ under $i^{\Ga}_{\Ga_0}$ by $(\mM\rtimes^{\Ga_0}\Ga, \be)$.
\end{lem}
\begin{proof}
    We define $\mM\rtimes^{\Ga_0}\Ga$ as follows. Objects are finite formal direct sums $\oplus_{\ga\in \Ga}(m_\ga,\ga)$ where $m_\ga\in\mM$ and only finitely many $m_\ga$ are nonzero. 
We define 
    \[
    \Hom((m,\ga),(m',\ga'))=\left\{\begin{array}{ll} \Hom_{\mM}(m\otimes\beta_0(\ga(\ga')^{-1}),m') &\quad  \mbox{ if } \ga(\ga')^{-1}\in \Ga_0,\\ 0 & \quad  \mbox{ oterwise.} \end{array}\right. 
    \] 
    The monoidal structure is given by
    \[
    (m,\gamma)\otimes (m',\gamma')=(m\otimes \alpha(\gamma)(m'), \gamma\gamma').
    \]
    Finally, $\beta: \Gamma\to \mM\rtimes^{\Ga_0}\Ga$ sends $\gamma$ to $(1_{\mM},\gamma)$.

    We construct the unit map $u:\id_{\bfA_{\Ga,\Ga_0}}\to r^{\Ga}_{\Ga_0}\circ i^{\Ga}_{\Ga_0}$ as follows. For $(\mM,\al,\be_0, \iota)\in \bfA_{\Ga,\Ga_0}$, $u(\mM,\al,\be_0, \iota)$ is the $\Ga$-equivariant functor $\mM\to \mM\rtimes^{\Ga_0}\Ga$ sending $m\in \mM$ to $(m,1)\in \mM\rtimes^{\Ga_0}\Ga$.     
    
    We construct the co-unit map $c:i^{\Ga}_{\Ga_0}\circ r^{\Ga}_{\Ga_0}\to \id_{\bfB_{\Ga}}$ as follows. For $(\mN, \be)\in \bfB_{\Ga}$, $c(\mN,\be)$ is the functor $\mN\rtimes^{\Ga_0}\Ga\to \mN$ sending $(n,\ga)$ to $n\otimes \be(\ga)$. We omit the details of checking that $c$ and $u$ satisfy the conditions for adjunction.
\end{proof}

When $\Ga_0=\{1\}$, $\mM\rtimes^{\{1\}}\Ga$ was denoted by $\mathrm{Proj}_{\La,\Ga}\ltimes\mM$ in \cite[7.2.11]{DLYZ25} (and was called the smashed product there).

\begin{rem}
    The construction in this subsection
    should have a higher categorical version although we choose not to formulate it precisely here. The current version is sufficient for our discussions in the next subsection.
\end{rem}

\subsection{Reductive groups}

Next, we assume that $G$ is tamely ramified quasi-split connected reductive group over $F$. It is enough to prove Theorem~\ref{intro: main thm} for the monodromic version, as the equivariant case can be recovered from the monodromic version using \eqref{eq: equiv vs mon} and its analogue for $\mH^\flat_{\unip,\af}$.

We recall once again that the action of $\wt\Omega$ on $\mH_\mon^\circ$ is induced from the conjugation action of $\wt\Omega$ on $I^+_k\bs LG^\circ_k/I^+_k$ (resp. on $I^{\flat,+}_k\bs LG^{\flat,\circ}_k/I^{\flat,+}_k$). In addition, the map \eqref{eq:wtOmega to Hk} induces a monoidal functor
\begin{equation}\label{eq: wtOmega to H}
\Shv(\wt\Omega)\to \Shv_{\mon}(I^+_k\bs LG_k/I^+_k),
\end{equation}
which sends $\La_{\tilde\omega}$ (the skyscraper sheaf supported at $\tilde\omega$) to the cofree monodromic standard sheaf on $I^+_k\bs I^+_k\tilde\omega/I^+_k$ trivialized at the point $\tilde\omega$. It restricts to a monoidal functor
\begin{equation}\label{eq: wtOmegacirc to Hcirc}
\Shv(\wt\Omega^{\circ})\to \Shv_\mon(I^+_k\bs LG^\circ_k/I^+_k)
\end{equation}
which is $\wt\Omega\rtimes\langle\sigma\rangle$-equivariant, by \eqref{eq: all in one diagram}. The functors \eqref{eq: wtOmega to H} and \eqref{eq: wtOmegacirc to Hcirc} have obvious equal characteristic counterparts.

We apply the general construction from \S\ref{ss:semidirect} as follows. On the one hand, 
we let $\mT il_{\mon,\af}^{\circ}\subset \mH_{\mon,\af}^{\circ}$ (resp. $\mT il_{\mon,\af}^{\flat,\circ}\subset \mH_{\mon,\af}^{\flat,\circ}$) be the full subcategory of cofree monodromic tilting sheaves as defined in \cite[\S 5]{DLYZ25} (for mixed characteristic, see \cite[\S 4.2]{Zhu25}). It is an ordinary monoidal additive category. 
The $\wt\Omega\rtimes\langle\sigma\rangle$-actions preserve these subcategories and therefore we obtain a $\wt\Omega\rtimes\langle\sigma\rangle$-equivariant equivalence
\begin{equation*}\label{eq: equiv for tiltings}
\mT il_{\mon,\af}^{\circ}\cong \mT il_{\mon,\af}^{\flat,\circ}
\end{equation*}
In addition, \eqref{eq: wtOmegacirc to Hcirc} restricts to a monoidal functor $\wt\Omega^\circ\to \mT il_{\mon,\af}^\circ$. Thus we obtain a $\sigma_*$-equivariant equivalence
\begin{equation}\label{eq: final equiv-1}
\mT il_{\mon,\af}^{\circ}\rtimes^{\wt\Omega^\circ} \wt\Omega\cong \mT il_{\mon,\af}^{\flat, \circ}\rtimes^{\wt\Omega^\circ}\wt\Omega. 
\end{equation}

On the other hand, we also have full subcategories $\mT il_{\mon,\af}\subset \mH_{\mon,\af}$ and $\mT il_{\mon,\af}^\flat\subset \mH_{\mon,\af}^{\flat}$ of cofree monodromic tilting sheaves in the monodromic affine Hecke categories (with all blocks included), and \eqref{eq: wtOmega to H} restricts to a monoidal functor $
\wt\Omega\to \mT il_{\mon,\af}$. 

Then Lemma \ref{lem: adjunction} provides a $\sigma_*$-equivariant monoidal functor
\begin{equation}\label{eq: final equiv-2}
\mT il_{\mon,\af}^{\circ}\rtimes^{\wt\Omega^\circ} \wt\Omega\to \mT il_{\mon,\af},
\end{equation}
which can be checked to be an equivalence by looking at each connected component of $LG_k$, which can then be translated to the neutral component by an element of $\wt\Omega$. Repeat the same arguments in the equal characteristic case gives a $\sigma_*$-equivariant monoidal equivalence
\begin{equation}\label{eq: final equiv-3}
\mT il_{\mon,\af}^{\flat,\circ}\rtimes^{\wt\Omega^\circ} \wt\Omega\to \mT il^{\flat}_{\mon,\af}.
\end{equation}

Combining \eqref{eq: final equiv-1}-\eqref{eq: final equiv-3}, we obtain a monoidal equivalence
\begin{equation}\label{eq: tilt equiv}
    \mT il_{\mon,\af}\cong \mT il_{\mon,\af}^{\flat}.
\end{equation}

Finally, as explained in \cite[Theorem 5.28]{DLYZ25}, $\mH_{\mon,\af}$ (resp. $\mH_{\mon,\af}^{\flat}$) can be recovered as a monoidal category from the additive monoidal category $\mT il_{\mon,\af}$ (resp. $\mT il_{\mon,\af}^\flat$). Indeed in \cite[Theorem 5.28]{DLYZ25} we worked with a fixed $W$-orbit $\Xi$ of mod $\ell$ reductions of character sheaves on $\mS_k$. We denote the resulting subcategory of $\mH_{\mon,\af}$ by $\mH_{\Ximon,\af}$, which can be recovered from its full subcategory of cofree tilting sheaves $\mT il_{\Ximon,\af}$ as follows: $\mH_{\Ximon,\af}$ embeds full faithfully into $\Ind(K^b\mT il_{\Ximon,\af})$, whose image is the ind-completion of subcategory generated by standard (constructible) objects.  Now $\mH_{\mon,\af}=\bigoplus_{\Xi}\mH_{\Ximon,\af}$ and $\mT il_{\mon,\af}=\bigoplus_{\Xi}\mT il_{\Ximon,\af}$ as monoidal categories, with convolution defined summand by summand. Similar remarks apply to the equal characteristic version. 
The desired equivalence $\mH_{\mon,\af}\cong \mH_{\mon,\af}^{\flat}$ then follows from \eqref{eq: tilt equiv}.

\address{
Department of Mathematics, Massachusetts Institute of Technology\\
Cambridge, MA, USA.\\
\email{zyun@mit.edu}\\
}
\medskip

\address{
Department of Mathematics, Stanford Univeristy\\
Stanford, CA, USA.\\
\email{zhuxw@stanford.edu}}

\end{document}